\documentclass[graybox]{svmult}

\usepackage{amssymb}
\usepackage{mathrsfs}
\usepackage{amsmath}
\usepackage{url}
\usepackage{stackrel,color} 
\usepackage{courier}
\usepackage{graphicx}
\usepackage[color,matrix,arrow,all,cmtip]{xy} 
\usepackage{multicol} 
\usepackage{verbatim} %  useful for \begin{comment} \end{comment}  
\usepackage{hyperref}
\usepackage{multicol}  
\usepackage[dvipsnames]{xcolor}
\usepackage{enumerate} 
\usepackage{bookmark} 
\catcode`~=11 % make LaTeX treat tilde (~) like a normal character
\newcommand{\urltilde}{\kern -.15em\lower .7ex\hbox{~}\kern .04em}  
\catcode`~=13 % revert back to treating tilde (~) as an active character

\makeatletter
\newcommand*\bigcdot{\mathpalette\bigcdot@{2.0}}
\newcommand*\bigcdot@[2]{\mathbin{\vcenter{\hbox{\scalebox{#2}{$\m@th#1\bullet$}}}}}
\makeatother

\renewcommand{\star}{{}^{*}}

\newcommand{\Stab}{\mathop{\mathrm{Stab}}\nolimits}

\newcommand{\proj}{\mathop{\mathrm{proj}}\nolimits}

\newcommand{\Aut}{\mathop{\mathrm{Aut}}\nolimits}

\newcommand{\Res}{\mathop{\mathrm{Res}}\nolimits}
\newcommand{\Ind}{\mathop{\mathrm{Ind}}\nolimits}

\newcommand{\Rep}{\mathop{\mathrm{Rep}}\nolimits}

\makeindex             % used for the subject index
\allowdisplaybreaks 

\begin{document}
\title*{Generalized iterated wreath products of symmetric groups and generalized rooted trees correspondence}
% Use \titlerunning{Short Title} for an abbreviated version of
% your contribution title if the original one is too long
\author{Mee Seong Im and Angela Wu} 
% Use \authorrunning{Short Title} for an abbreviated version of
% your contribution title if the original one is too long
\institute{Mee Seong Im \at Department of Mathematical Sciences, United States Military Academy, West Point, NY 10996 USA  \newline 
\email{meeseongim@gmail.com}
\and Angela Wu \at Department of Mathematics, University of Chicago, Chicago, IL 60637 USA
\newline 
 \email{wu@math.uchicago.edu}}
%
% Use the package "url.sty" to avoid
% problems with special characters
% used in your e-mail or web address
%
\titlerunning{Iterated wreath products of symmetric groups and rooted trees}% Short title 
\maketitle

%\abstract*{Consider the generalized iterated wreath product $\mathbb{Z}_{r_1}\wr \mathbb{Z}_{r_2}\wr \ldots \wr \mathbb{Z}_{r_k}$ where $r_i \in \mathbb{N}$.   
%We prove that the irreducible representations for this class of groups are indexed by a certain type of rooted trees. 
%This provides a Bratteli diagram for the generalized iterated wreath product, a simple recursion formula for the number of irreducible representations, and a strategy to calculate the dimension of each irreducible representation. We calculate explicitly fast Fourier transforms (FFT) for this class of groups, giving literature's fastest FFT upper bound estimate. }

\abstract{Consider the generalized iterated wreath product $S_{r_1}\wr \ldots \wr S_{r_k}$ of symmetric groups. We give a complete description of the traversal for the generalized iterated wreath product. We also prove an existence of a bijection between the equivalence classes of ordinary irreducible representations of the generalized iterated wreath product and orbits of labels on certain rooted trees. We find a recursion for the number of these labels and the degrees of irreducible representations of the generalized iterated wreath product. Finally, we give rough upper bound estimates for fast Fourier transforms.  
% \subjclass[2000]{Primary 20C99, 20E08; Secondary 65T50, 05E25}
}

\section{Introduction}\label{section:introduction}

The representation theory of the symmetric group is remarkably prevalent in combinatorics; one can explicitly parametrize the irreducible representations of the symmetric group using Young diagrams, leading us to the study of the interaction of these diagrams, an examination of the decomposition of tensor products of Young diagrams, and an investigation of the dimension of the irreducible representation associated to a Young diagram. 
Wreath products of symmetric groups arise as the automorphism group of regular rooted trees (see Theorem 2.1.6 or Theorem 2.1.15 in \cite{MR3202374}), with applications ranging from functions on rooted trees (see Section~\ref{subsection:Bratteli}), pixel blurring  (cf. \cite{Stankovic-Moraga-Astola}, \cite{Chang-Thesis},  \cite{Mirchandani99multiresolution-analysis}, \cite{Mirchandani99awreath}, \cite{Holmes-mathematical-foundations}, \cite{Holmes-signal-processing}), symmetries of nonrigid molecules in molecular spectroscopy (cf. \cite{balasubramanian1979enumeration}, \cite{MR585739}, \cite{milot2001energy}, \cite{schnell2010understanding}), and visual information processing (cf. \cite{borsa2015wreath}, \cite{leyton2003generative}) 
 to choosing subcommittees from sets of committees and voting (cf.  \cite{MR3625572}, \cite{MR2572103}, \cite{Lee-Stephen}).  
With motivation from  \cite{Branching-diag-cyclic-Im-Wu} and \cite{MR2081042},   
we consider generalized iterated wreath product $W(\mathbf{r}|_k):= S_{r_1}\wr \ldots \wr S_{r_k}$ of symmetric groups, where $S_{r_i}$ is the symmetric group on $r_i$ letters, and study its representation theory.

% Throughout this manuscript, assume that $G$ is a finite group. 
Throughout this manuscript, let $G$ be a finite group, and let $V$ be a vector space over the complex numbers $\mathbb{C}$. Let $GL(V)$ be the general linear group on $V$, and let $\rho:G\rightarrow GL(V)$ be a representation of $G$, i.e., $\rho$ is a group homomorphism. 
We say two representations  $\rho:G\rightarrow GL(V)$ and $\eta:G\rightarrow GL(W)$   are equivalent, and write $\rho \sim \eta$, if there exists a vector space isomorphism $f:V\rightarrow W$ such that $f\circ \rho(g)=\eta(g)\circ f$ for all $g\in G$. 
 We denote by $\widehat{G}$ the set of irreducible representations of $G$.   
We say that  
$\mathcal{R}$ is a \textit{traversal} for $G$ if $\mathcal{R} := \mathcal{R}_G \subset \widehat{G}$ contains exactly one irreducible representation for each isomorphism class in $\widehat{G}$. Thus a traversal consists of a complete list of pairwise inequivalent irreducible representations of $G$. 
As a basic consequence of representation theory, the equality $\sum_{\rho \in \mathcal{R}} \dim(\rho)^2 = |G|$ holds, where the sum is over all irreducible representations in $\mathcal{R}$.

Let $[n]:=\{1,2,\ldots, n \}$, the set of integers from $1$ to $n$, and let 
$$
[n]^{\ell}:=\{ x_1 x_2\cdots x_{\ell}: x_i\in [n] \}, 
$$ 
 the set of length $\ell$ words with letters in $[n]$.

Now given a subgroup $H \leq G$, we write $\Ind_H^G:\Rep(H)\rightarrow \Rep(G)$ to be the induction functor from the category of representations of $H$ to the category of representations of $G$. 
That is, given a representation $\eta\in \widehat{H}$ of subgroup $H\subseteq G$, where $\eta:H\rightarrow GL(V)$, 
 we write $\Ind_H^G \eta = \mathbb{C}[G]\otimes_{\mathbb{C}[H]}V$, the {\it induced representation} of $G$ from $\eta$ with dimension $[G:H] \cdot \dim \eta$. 
There also exists the dual construction to induction called restriction.  Given a subgroup $H$ of $G$,
$\Res_H^G:\Rep(G)\rightarrow \Rep(H)$ is the restriction functor from the category of representations of $G$ to the category of representations of $H$, i.e., 
given a representation $\rho$ of $G$, we obtain the {\em restricted representation} $\Res_H^G\rho$ of $H$ by restricting $\rho$ to $H$. 
The induction and restriction functors are related by Frobenius reciprocity. 
We refer the reader to \cite{MR0202859} for a detailed and elegant discussion on the duality of the induction and restriction functors.

	We say that $\mathbf{\alpha} = (\alpha_1, \ldots, \alpha_h)$ is a \textit{partition} of a natural number $n>0$, and write $\mathbf{\alpha} \vdash n$, 
	if every $\alpha_i$ satisfies the following: 
	\begin{enumerate}
	\item for each $i$, $\alpha_i \in \mathbb{N}=\{ 1,2,3,\ldots\}$, 
	\item $\alpha_1\geq \alpha_2\geq \ldots \geq \alpha_h$, and 
	\item $\sum_{i=1}^h \alpha_i = n$.  
	\end{enumerate}
	If $\alpha\vdash n$, then we will denote the length of $\alpha$ by $|\mathbf{\alpha}|= h$.
We will write $\alpha \models_h n$ to denote that $\alpha = (\alpha_1,\ldots, \alpha_h) \in \left( \mathbb{Z}_{\geq 0}\right)^h$ is  a {\it weak composition} of the natural number $n$ with $h$ parts so that each $\alpha_i\geq 0$ is an integer and $\sum_{i=1}^h \alpha_i = n$.

For 
$\mathbf{\alpha} \models_h n$,  we define 
$$
S_{\mathbf{\alpha}}:= S_{\alpha_1} \times S_{\alpha_2} \times \cdots \times S_{\alpha_h} \leq S_n, 
$$
the permutation subgroup acting on the set $[n]$ by the full action on $h$ disjoint orbits of size $\alpha_j$. 
We also define $S_0 = \{ 1\}$.

For a group $G$, 
we will now discuss inner and outer tensor products associated to describing the irreducible representations of $G\wr S_n$. 
If $\rho$ and $\eta$ are representations of $G$, then their inner tensor product $\rho\otimes \eta$ is again a representation of $G$ defined by $(\rho\otimes \eta)(g)= \rho(g)\otimes \eta(g)$, where $g\in G$. 
If $\rho$ is a representation of $G$ and $\eta$ is a representation of a group $H$, then their outer tensor product $\rho\boxtimes\eta$ is a representation of $G\times H$ defined by 
$(\rho\boxtimes\eta)(g,h)= \rho(g)\otimes \eta(h)$, where $g\in G$ and $h\in H$.

 The irreducible representations of the base group $G^n=G\times \cdots \times G$ are $n$-fold outer tensor products of irreducible representations of $G$.  
 If $\mathcal{R}=\{ \rho_1,\ldots,\rho_h \}$ is a traversal for $G$ and $\alpha\models_h n$, then 
 $\rho_1^{\boxtimes \alpha_1}\boxtimes \cdots \boxtimes \rho_h^{\boxtimes \alpha_h}$ is an irreducible representation of $G^n$, which can be extended to an irreducible representation $(\rho_1^{\boxtimes \alpha_1}\boxtimes \cdots \boxtimes \rho_h^{\boxtimes \alpha_h})'$ of the inertia group $G\wr S_{\alpha}$. 
 On the other hand, if $\sigma\in  \mathcal{R}_{S_{\alpha}}$, 
 then composing $\sigma$ with the projection of $G\wr S_{\alpha}$ onto $S_{\alpha}$ gives an irreducible representation $\sigma'$ of $G\wr S_{\alpha}$. 
 Now, the inner tensor product of  $(\rho_1^{\boxtimes \alpha_1}\boxtimes \cdots \boxtimes \rho_h^{\boxtimes \alpha_h})'$ and $\sigma'$ is an irreducible representation of $G\wr S_{\alpha}$, and the induced representation $ \Ind_{G\wr S_\alpha}^{G \wr S_n} ((\rho_1^{\boxtimes \alpha_1} \boxtimes \cdots \boxtimes \rho_h^{\boxtimes \alpha_h})' \otimes \sigma')$ is an irreducible representation of $G\wr S_n$. 
 With these remarks, 
 we give an explicit description of the traversal of $W(\mathbf{r}|_k)$: 
	\begin{theorem}\label{thm:traversal}
For $N>0$, 
let $\mathcal{R}_G = \{ \rho_1, \ldots , \rho_h \}$ be a traversal for a group $G \leq S_N$. Let $\mathbf{\alpha} \models_h n$. 
		 Then the irreducible representations given by
		 \begin{equation}\label{equation:traversal-wreath-full-symmetric}
		 \left\{ \Ind_{G\wr S_\alpha}^{G \wr S_n} ((\rho_1^{\boxtimes \alpha_1} \boxtimes \cdots \boxtimes \rho_h^{\boxtimes \alpha_h})' \otimes \sigma') : \alpha \models_h n, \sigma \in \mathcal{R}_{{S}_{\alpha}} \right\}
		 \end{equation}
		 form a traversal for $G\wr S_n$.  
		 In particular, if $\mathcal{R}_{W(\mathbf{r}|_{k-1})} = \{ \rho_1 ,\ldots , \rho_h\}$ is a traversal for the wreath product $W(\mathbf{r}|_{k-1})$, then a traversal for $W(\mathbf{r}|_k)$ is 
				\begin{equation}		
				\label{eq_set-of-irreps}	 
\mathcal{R}_{W(\mathbf{r}|_{k})} = \left\{ \Ind_{W(\mathbf{r}|_{k-1}) \wr S_\alpha}^{W(\mathbf{r}|_k)} ((\rho_1^{\boxtimes \alpha_1} \boxtimes \cdots \boxtimes \rho_h^{\boxtimes \alpha_h})' \otimes \sigma') :  \alpha \models_h r_k, \sigma \in \mathcal{R}_{S_{\alpha}} \right\},  
		 		 \end{equation}
		 		 where $S_{\alpha}$ is a subgroup of $S_{r_k}$. 
	\end{theorem}

Note that we write $\rho^\alpha := \rho_1^{\boxtimes \alpha_1} \boxtimes \cdots \boxtimes \rho_h^{\boxtimes \alpha_h}$, where $\rho_1,\ldots, \rho_h$ are traversals of a group $G$, and we define $\rho^0 := 1$.

We also find a recursion for the number of equivalence classes of ordinary irreducible representations of the generalized iterated wreath products in Corollary~\ref{cor:recursion-number-of-irrep}, and their dimensions are given in Proposition~\ref{prop:dimension-irreducible-rep}.

A {\it rooted tree} is a connected simple graph with no cycles, 
and with a distinguished vertex, which is called the {\it root}. 	
We refer to Section~\ref{subsection:rooted-trees} for a further discussion on rooted trees. 	
We recall the following theorem: 
\begin{theorem}[Theorem 2.1.15, \cite{MR3202374}]\label{thm:geometric-construction-tree}
Let $\mathcal{T}(\mathbf{r}|_k)$ be a complete $\mathbf{r}$-tree of height $k$. 
We have 
\[
\Aut(\mathcal{T}(\mathbf{r}|_k)) \cong W(\mathbf{r}|_k). 
\] 
\end{theorem}

We find a bijection between equivalence classes of ordinary irreducible representations of the generalized iterated wreath product  
$W(\mathbf{r}|_k)$ and the orbits of families of labels on certain complete trees, thus connecting to the geometric construction in Theorem~\ref{thm:geometric-construction-tree}: 
\begin{theorem}\label{theorem:bijection-tree-iterated-wreath}
There is a bijection between equivalence classes $\widehat{W}(\mathbf{r}|_k)$ of ordinary irreducible representations of the iterated wreath product of symmetric groups and $W(\mathbf{r}|_k)$-orbits of rooted trees $\mathcal{T}( \mathbf{r}|_k)$. 
\end{theorem}

\subsection{Summary of the sections}\label{subsection:summary}

We begin Section~\ref{section: background} with some background.  
We give a summary of the representation theory of symmetric groups in Section~\ref{subsection:rep-theory-symmetric-group},  give the construction of iterated wreath products in Section~\ref{subsection:wreath-products}, and discuss Clifford theory in Section~\ref{subsection:Clifford-thy}. We give the construction of rooted trees in Section~\ref{subsection:rooted-trees}, and Bratteli diagrams in Section~\ref{subsection:Bratteli}. We conclude the background section by reviewing adapted bases and fast Fourier transforms  in Section~\ref{subsection:adaped-bases-FFT}. 
We prove Theorem~\ref{thm:traversal} in Section~\ref{section:irrep-iterated-wreath-products}, and we prove Theorem~\ref{theorem:bijection-tree-iterated-wreath} in Section~\ref{section:branching-diagram}. 
 We also give the dimension of an irreducible representation of the iterated wreath product in Proposition~\ref{prop:dimension-irreducible-rep}. 
 In Section~\ref{section:FFT}, we give coarse upper bound estimates for fast Fourier transforms for $G\wr S_n$ (and thus for the generalized iterated wreath product $W(\mathbf{r}|_k)$ in Corollary~\ref{cor:FFT-symmetric-groups}), and in Section~\ref{section:conclusion-open-problem}, we discuss some open problems.

\subsection{Acknowledgment}    
The authors acknowledge Mathematics Research Communities for providing an exceptional working environment at Snowbird, Utah. They would like to thank Michael Orrison for helpful discussions, and the referees for immensely valuable comments. 
This paper was written during MSI's visit to the University of Chicago in 2014. She thanks their hospitality.

\section{Background}\label{section: background}
 
We will begin by giving some necessary background. 

\subsection{Representations of the symmetric group}\label{subsection:rep-theory-symmetric-group}

We refer to \cite{MR2643487}, \cite{MR1153249}, \cite{MR644144}, \cite{MR0325752}, and \cite{kleshchev2010representation}  for an extensive background on representations of symmetric groups. In this section, we will give a brief summary of the representation theory of symmetric groups. 

The symmetric group $S_n$ has order $n!$ whose conjugacy classes are labeled by partitions of $n$. Thus, the number of inequivalent irreducible representations over $\mathbb{C}$ is equal to the number of partitions of $n$. One may also parametrize irreducible representations by the same set that parametrizes conjugacy classes for $S_n$, 
which is by partitions of $n$, or, equivalently, the more commonly used of so-called Young diagrams of size $n$ (see Example~\ref{example:Bratteli-diagram}). 

\subsection{Wreath products}\label{subsection:wreath-products}
We refer to Chapter 2 in \cite{MR3202374} for a beautiful exposition on the construction of the wreath product $G\wr H$ of a finite group $G$ with a subgroup $H\leq S_n$, which is summarized as follows. Define an action of $H$ on $G^n=G\times \cdots \times G$
 by if $\pi\in H$ and $a=(a_1,a_2,\ldots, a_n)\in G^n$, then 
$\pi\cdot a := a^{\pi}=(a_{\pi^{-1}(1)},a_{\pi^{-1}(2)},\ldots, a_{\pi^{-1}(n)})$. 
The wreath product $G\wr H$ is defined to be $G^n\times H$ as a set, with multiplication given by 
\begin{equation}\label{eqn:multn-wreath-product}
(a;\pi)(b;\sigma) = (ab^{\pi};\pi\sigma). 
\end{equation}

Throughout this paper, we will fix $\mathbf{r} = (r_1,r_2,r_3,\ldots) \in \mathbb{N}^\omega$, a positive integral vector. We denote by $\mathbf{r}|_k :=(r_1,r_2,\ldots, r_k)$, the length $k$ vector found by truncating $\mathbf{r}$.

Let $H$ be a finite group. 
Let $H^X:= \{ f:X\rightarrow H\}$, a set of all maps from $X$ to $H$, which is a group under pointwise multiplication:   $(f\circ f')(x)=f(x)f'(x)$ for all $x\in X$.  
Now for $H$ acting on a set $X$ and $G$ acting on a set $Y$, 
\[
(g,h)^{-1}(x,y)=(h^{-1}g^{-1},h^{-1})(x,y) = (h^{-1}x,g(x)^{-1}y) 
\] 
for all $(g,h)\in G\wr H$ and $x\in X$ and $y\in Y$.

\begin{definition}\label{definition:chain-of-subgroups}

Let $S_{r_i}$ be a symmetric group acting on a finite set of order $r_i$ for every $1\leq i\leq k$. Assume $S_{r_i}$ acts on a finite set $X_i$, where 
$1\leq i\leq k-1$. Set $V_{k+1}=\{ \varnothing\}$ and, 
for $i=1,2,\ldots, k$, let 
\[ 
V_i= X_i\times X_{i+1}\times \cdots \times X_{k-1} \times X_k. 
\] 
The {\em generalized iterated wreath product } $W(\mathbf{r}|_k):= S_{r_1}\wr\ldots \wr S_{r_k}$ of symmetric groups consists of all $k$-tuples 
$(g_1,\ldots, g_k)$, where $g_k\in S_{r_k}$, and 
$g_i:V_{i+1}\rightarrow S_{r_i}$, $1 \leq i< k$, with the multiplication law and action on $V_{1}$ recursively defined in the following way: 
\[ 
\begin{split}
(g_i,\ldots, g_k)&(g_i',\ldots, g_k') =  \\
\bigg(g_i\cdot (g_{i+1},\ldots, g_{k-1}, &g_k)g_i', 
      (g_{i+1},\ldots, g_{k-1}, g_k)(g_{i+1}',\ldots, g_{k-1}', g_k')\bigg), \\
\end{split}
\] 
where 
\[ 
\begin{split}
\bigg((g_{i+1},g_{i+2},\ldots, g_{k-1}, &g_k)g_i'\bigg)
 (x_{i+1},\ldots, x_{k-1},x_k)
= \\ 
g_i'\bigg((g_{i+1},&\ldots, g_{k-1}, g_k)^{-1} (x_{i+1},\ldots, x_{k-1},x_k)  \bigg), \\
\end{split}
\] 
and by 

\begin{equation}
\begin{split}
(g_{i+1}, g_{i+2},\ldots, &g_{k-1}, g_k) (x_{i+1},\ldots, x_{k-1},x_k) 
   = \\ 
  \bigg( 
(g_{i+2}, & 
\ldots, g_{k-1}, g_k)(x_{i+2},\ldots, x_{k-1},x_k), \\ 
&
g_{i+1}  (g_{i+2},\ldots, g_{k-1},g_k)(x_{i+2},\ldots, x_{k-1},x_k) x_{i+1}
\bigg) \\
\end{split}
\end{equation}
for all $x_j\in X_j$, $g_j\in S_{r_j}^{V_{j+1}}$,  
$i\leq j\leq k$ and $i=1,2,\ldots, k$. 
\end{definition}

\begin{remark}\label{remark:defn-iterated-wreath-symm}  
The generalized $k$-th $\mathbf{r}$-symmetric wreath product $W(\mathbf{r}|_k)$ 
could also be defined recursively by 
	\begin{equation*}
		W(\mathbf{r}|_0) = \{ 1 \} \text{ and } W(\mathbf{r}|_k) = W(\mathbf{r}|_{k-1}) \wr S_{r_k}, 
	\end{equation*}
	where the multiplication for the wreath product $W(\mathbf{r}|_k)$ is defined recursively using \eqref{eqn:multn-wreath-product}. 
 \end{remark}

\begin{example}\label{ex:wr-prod-symm}
Note that $W(\mathbf{r}|_1) = S_{r_1}$, $W(\mathbf{r}|_2) = S_{r_1} \wr S_{r_2}$, and $W(\mathbf{r}|_k) = S_{r_1} \wr S_{r_2}\wr \ldots \wr S_{r_k}$. 
\end{example}

Throughout this paper, we will be considering the chain of groups given in Remark~\ref{remark:defn-iterated-wreath-symm}.

\subsection{Clifford theory}\label{subsection:Clifford-thy}

The following references \cite{MR1038525}, \cite{MR1060103}, \cite{MR2760311}, and \cite{MR3202374} contain an extensive background on Clifford theory, which allow one to use recursion to construct the irreducible representations of a group. 
In this manuscript, 
we will give a brief overview of the main results of Clifford theory and the little-group method.

Let $G$ be a finite group and let $N \triangleleft G$ be a normal subgroup of $G$. For two representations $\sigma, \rho$, we write $\sigma \prec \rho$ if $\sigma$ is a {\it subrepresentation} of $\rho$. We say that $\widetilde{\sigma} \in \widehat{G}$ is an {\em extension} of $\sigma \in \widehat{N}$ if $\Res_N^G \widetilde{\sigma} = \sigma$. 

\begin{definition}
Fix $\theta \in \widehat{N}$ and $g \in G$. 
	\begin{enumerate}
	\item We define  		
		\begin{equation}
			\widehat{G}(\theta) := 
			\left\{ \rho \in \widehat{G} : \theta \prec \Res_{N}^G \rho \right\}. 
		\end{equation}
	\item The $g$-conjugate $\sigma^g \in \widehat{N}$ of $\sigma$ is defined as $\sigma^g(h) := \sigma(g h g^{-1})$ for any $h \in N$. 
	\item The inertia group of $\sigma$ in $G$ is given by $I_G(\sigma) := \{ g \in G : \sigma^g \sim \sigma \}$. 
	\end{enumerate} 
\end{definition}

Now, the finite group $G$ also acts on the set of inequivalent irreducible representations of $N$. For any irreducible representation $\sigma$ of $N$, let 
$\Delta(\sigma)$ denote its orbit under this action, i.e., inequivalent conjugates of $\sigma$. 
Let $\Stab_G(\sigma)$ be the isotropy subgroup of $\sigma$ under the $G$-action. 
\begin{theorem}[Clifford theory]\label{thm:Clifford-theory}
Let $N$ be a normal subgroup of $G$. 
\begin{enumerate}
\item(\cite{MR0202859}, Theorem 10)\label{item:Clifford-one} 
If $\sigma$ be a representation of $N$, then 
\[ 
\Res_N^G \Ind_N^G \sigma = [\Stab(\sigma):N]\cdot \Delta(\sigma). 
\] 
\item(\cite{MR0202859}, Theorem 14)\label{item:Clifford-two} 
If $\rho:G\rightarrow GL(V)$ is an irreducible representation of $G$, then 
\[ 
\Res_N^G \rho =  \frac{d_{\rho}}{[G:\Stab(\sigma)]d_{\sigma}}\cdot \Delta(\sigma), 
\] 
where $\sigma$ is any irreducible representation of $N$ appearing in $\Res_N^G \rho$, and 
$d_{\rho}$ is the dimension of the vector space $V$.  
\end{enumerate}
\end{theorem}
We also call $d_{\rho}$ the {\em degree} of the representation $\rho$. 
% page 5, Rockmore.  
Next, the little-group method provided below is motivated by Chapter 5 Section 1 in \cite{MR1363490}. 
\begin{theorem}[Little-group method]
\label{thm:little-group-method}
	Suppose that any $\sigma \in \widehat{N}$ has an extension $\widetilde{\sigma}$ to its inertia group $I_G(\sigma)$. Let $\Sigma$ be a set of orbit representatives of the irreducible representations of $N$ under action of $G$, where $g \in G$ acts on $\sigma \in \widehat{N}$ by $\sigma^g$. Then a traversal of $G$ is given by  
		\begin{equation*}
		\left\{ \Ind^G_{I_G(\sigma)} (\widetilde{\sigma} \otimes \bar{\psi} ): \sigma \in \Sigma, \psi \in \mathcal{R}_{I_G(\sigma)/N} \right\}, 
		\end{equation*}
where $\bar{\psi}$ is the representation on $I_G(\sigma)$ given by $\bar{\psi}(g) = \psi(\proj(g))$, where $\proj: I_G(\sigma) \rightarrow I_G(\sigma)/N$ is a canonical projection map. 
\end{theorem}

\begin{lemma}
 	Suppose that $\{ \rho_1 , \ldots , \rho_h \}$ is a traversal for $G$. Then $\{ \rho(\eta) : \eta \in [h]^n \}$ is a traversal for $G^n$, where $\rho(\eta) := \rho_{\eta_1} \boxtimes \cdots \boxtimes \rho_{\eta_n}$ and $[h]=\{ 1,2,\ldots,h\}$.  
\end{lemma}

\begin{definition}\label{def:action-on-irreps}
The permutation action of $S_n$ on $G^n$ by permuting the factors of $G^n$ induces an action of $S_n$ on $\widehat{G^n}$ by 
$$
(\rho(\eta))^\sigma(g_1, \ldots, g_n) = \rho(\eta)(g_{\sigma^{-1}(1)}, \ldots , g_{\sigma^{-1}(n)}).
$$  
\end{definition}

It follows from Definition~\ref{def:action-on-irreps} that $\rho(\eta^\sigma) \sim \rho(\eta)$.

\begin{lemma}
	\label{lemma:orbits-reps-of-GN}
	For any $\eta, \mu \in [h]^n$, $\rho(\eta) \sim \rho(\mu)$ if and only if 
	$$
	\left| \{ j \in [n]:  \eta_j = \ell \} \right| = \left| \{ j \in [n]:  \mu_j = \ell \} \right|
	$$ 
	for any $\ell \in [h]$. 
	Thus the set $\{ \rho^\alpha : \alpha \models_h n\}$ forms a complete set of representatives for the orbits of $\widehat{G^n}$ under action by $S_n$. 
\end{lemma}

\subsection{Rooted trees of a fixed height}\label{subsection:rooted-trees}

Let $\mathbf{r}= (r_1,r_2,r_3, \ldots)\in \mathbb{Z}_{\geq 0}^{\mathbb{N}}$.  
In this section, we will give the construction of $\mathbf{r}$-rooted trees, generalizing the $r$-trees discussed in Section 3 of \cite{MR2081042}. 
A {\em rooted tree} is a connected simple graph with no cycles and with a distinguished vertex, which we call a {\em root}. 
We say a node $v$, i.e., a vertex, is in the  
{\em $j$-th layer} of a rooted tree if it is at distance $j$ from the root. 
The {\em branching factor} of a vertex is its number of children, and a {\em leaf} is a vertex with branching factor zero. 
% We will write $v_{root}$ to denote the root node. 

 \begin{definition}
We define the complete $\mathbf{r}$-tree $\mathcal{T}(\mathbf{r}|_k)$ of height $k$, or $\mathbf{r}|_k$-tree, recursively as follows. Let $\mathcal{T}(r_1)$ be the tree consisting of a root node only. Let $\mathcal{T}(\mathbf{r}|_k)$ consist of a root node with $r_k$ children, 
with each the vertex in the first layer a copy of the $k-1$-level tree $\mathcal{T}(\mathbf{r}|_{k-1})$, which yields a tree with $k$ levels of nodes.  
\end{definition}

We will also denote the complete $\mathbf{r}$-tree by $\mathbf{r}|_k$-tree.

\begin{example}
The tree $\mathcal{T}(r_1)$ is given by $\bullet$ and $\mathcal{T}(r_1,r_2)$ with $r_2$ leaves is given by 
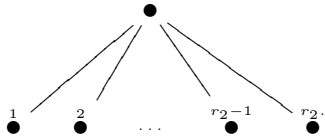
\tiny 
\[ 
\xymatrix@-1pc{ 
& & \ar@{-}[lldd] \ar@{-}[ldd]\bigcdot \ar@{-}[rdd] \ar@{-}[rrdd]& &  \\ 
& &  & &  \\ 
\stackrel{1}{\bigcdot} & \stackrel{2}{\bigcdot} & \ldots & \stackrel{r_2-1}{\bigcdot} &\stackrel{r_2.}{\bigcdot} \\ 
}
\] 
\end{example}

\begin{example}\label{example:complete-tree-3-layers-nodes} 
Writing $\mathbf{r}|_3= (r_1,r_2,r_3)$, $\mathbf{r}|_3$-tree of height $3$ with $3$ levels of nodes is given by  
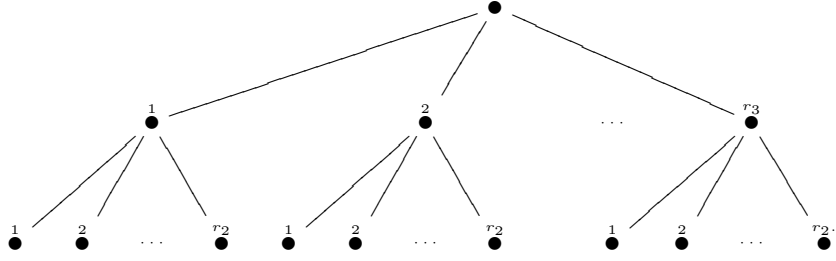
\tiny    
\[  
\xymatrix@-1pc{  
& & & & & & & \ar@{-}[llllldd]  \ar@{-}[ldd]\bigcdot  \ar@{-}[rrrrdd]& &  &  & & & &  \\   
& & & & & & & &  & & & & & & \\   
& & \ar@{-}[lldd] \ar@{-}[ldd] \stackrel{1}{\bigcdot}\ar@{-}[rdd] & & & &\ar@{-}[lldd] \ar@{-}[ldd]\stackrel{2}{\bigcdot} \ar@{-}[rdd]& & & \cdots&    &  \ar@{-}[lldd] \ar@{-}[ldd]  \stackrel{r_3}{\bigcdot} \ar@{-}[rdd] & & &  \\ 
& &  & & & & &  & & & & & & & \\  
\stackrel{1}{\bigcdot} & \stackrel{2}{\bigcdot} & \cdots &\stackrel{r_2}{\bigcdot} &  \stackrel{1}{\bigcdot}&  \stackrel{2}{\bigcdot} & \cdots & \stackrel{r_2}{\bigcdot}&   & \stackrel{1}{\bigcdot} & \stackrel{2}{\bigcdot} & \cdots &\stackrel{r_2.}{\bigcdot}  & &  &\\  
}
\] 
\end{example}

\begin{example}\label{ex:4-layers-node}
The complete tree $\mathcal{T}(\mathbf{r}|_4)$ of height $4$ with $4$ levels of nodes is given by 
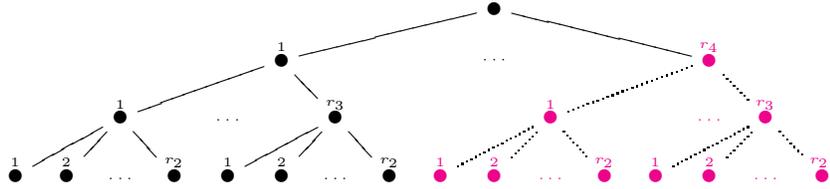
\tiny 
\[ 
\xymatrix@-1.5pc{  
& & & & & & & & & \ar@{-}[dllll]   \bigcdot   \ar@{-}[drrrr]& & & &  & & & &    \\  
& & & & & \ar@{-}[dlll] \stackrel{1}{\bigcdot} \ar@{-}[dr] & &  & &\cdots &  & & & \color{magenta}\ar@{..}@[magenta][dlll]\stackrel{r_4}{\bigcdot} \ar@{..}@[magenta][dr]\color{black}&  &  & &   \\ 
& & \ar@{-}[dll] \ar@{-}[dl]\stackrel{1}{\bigcdot}\ar@{-}[dr]&  &\ldots  & &  \ar@{-}[dll] \ar@{-}[dl]\stackrel{r_3}{\bigcdot}\ar@{-}[dr] &  & &  &  \ar@{..}@[magenta][dll] \ar@{..}@[magenta][dl]\color{magenta}\stackrel{1}{\bigcdot}\color{black} \ar@{..}@[magenta][dr]& & &\color{magenta}\ldots &  \ar@{..}@[magenta][dll] \ar@{..}@[magenta][dl] \color{magenta}\stackrel{r_3}{\bigcdot} \color{black}\ar@{..}@[magenta][dr]& &  &   \\ 
 \stackrel{1}{\bigcdot}& \stackrel{2}{\bigcdot}& \ldots & \stackrel{r_2}{\bigcdot} & \stackrel{1}{\bigcdot}& \stackrel{2}{\bigcdot}& \ldots &\stackrel{r_2}{\bigcdot}& \color{magenta}  \stackrel{1}{\bigcdot}\color{black}&\color{magenta} \stackrel{2}{\bigcdot}&\color{magenta} \ldots & \color{magenta}\stackrel{r_2}{\bigcdot} & \color{magenta}    \stackrel{1}{\bigcdot}&\color{magenta} \stackrel{2}{\bigcdot}& \color{magenta}\ldots &\color{magenta}\stackrel{r_2.}{\bigcdot} & &     \\  
}
\] 
\end{example}

% Make sure that this is correct! 
Notice that $\mathcal{T}(\mathbf{r}|_k)$ has $\displaystyle{\prod_{i=2}^k } r_i$ leaves, with 
$\displaystyle{\prod_{i=k-j+1}^k r_i}$ nodes in the $j$-th layer. The subtree $\mathcal{T}_v$ of $\mathcal{T}=\mathcal{T}(\mathbf{r}|_k)$ is the tree rooted at $v$ consisting of all the children and descendants of $v$. We call $\mathcal{T}_v$ a {\em maximal subtree} of $\mathcal{T}$ if $v$ is a child of the root, or equivalently if $v$ is in the first layer. Let $\deg(v)$ denote the number of leaves of the subtree $\mathcal{T}_v$.  
 
\begin{example}
In Example~\ref{ex:4-layers-node}, the subtree indicated by dotted edges in magenta is a maximal subtree of  $\mathcal{T}(\mathbf{r}|_4)$. 
\end{example}

We define 
\begin{equation}
\widehat{S}_* := \bigsqcup_{n \in \mathbb{N}}  \bigsqcup_{\alpha \vdash n} \widehat{S}_\alpha. 
\end{equation}
 
\begin{definition}
An {\em $\mathbf{r}|_k$-label }is a function $\phi: V_{\mathcal{T}(\mathbf{r}|_k)} \rightarrow \widehat{S}_*$ on the vertices $ V_{\mathcal{T}(\mathbf{r}|_k)}$ of the tree $\mathcal{T}(\mathbf{r}|_k)$
satisfying 
$$
\phi(v) \in \bigsqcup_{\alpha \vdash \deg(v)} \widehat{S}_\alpha. 
$$  

%A {\em labeled $\mathbf{r}$-tree of height $k$  } is a pair  $(T(\mathbf{r}|_k), \phi)$ consisting of a tree and a compatible label of $T(\mathbf{r}|_k)$, where a {\em compatible label of $T(\mathbf{r}|_k)$} is a function $\phi: V_{T(\mathbf{r}|_k)} \rightarrow \widehat{S}_* := \displaystyle{\sqcup_{\alpha \vdash n} S_\alpha}$ assigning each node $v$ of $T(\mathbf{r}|_k)$ to an irreducible representation of $S_\alpha$ for some $\alpha \vdash [\deg(v)]$. 
We say that two labels $\phi$ and $\psi$ on $\mathcal{T}(\mathbf{r}|_k)$ are {\em equivalent}, and write $\phi \sim \psi$, if there exists $\sigma \in \Aut(\mathcal{T}(\mathbf{r}|_k))$ such that $\phi^\sigma = \psi$, where $\phi^\sigma(v) := \phi(v^\sigma)$, which is defined as the right action of $\sigma\in W(\mathbf{r}|_k)$ on $v$.    
\end{definition}  

In other words, two compatible labels $\phi$ and $\psi$ of $\mathcal{T}(\mathbf{r}|_k)$ are equivalent if they are in the same orbit under the $W(\mathbf{r}|_k)$-action, or equivalently,  
if $\phi^{W(\mathbf{r}|_k)} = \psi^{W(\mathbf{r}|_k)}$.

\begin{definition}   
An $\mathbf{r}|_k$-label $\phi: V_{\mathcal{T}(\mathbf{r}|_k)} \rightarrow \widehat{S}_* $ is {\em valid}  
if it satisfies all of the following recursive conditions. We denote by 
$$
\mathcal{T}(\mathbf{r}|_k)
%(\mathbf{r}|_k) 
:= \left\{ \phi: \phi \text{ is a valid $\mathbf{r}|_k$-label}\right\} 
\hspace{4mm}
\mbox{ and } 
\hspace{4mm}
\mathcal{T} = \bigsqcup_k \mathcal{T}(\mathbf{r}|_k). 
$$  
	\begin{enumerate}
	\item Given an $\mathbf{r}|_1$-label $ \phi : V_{\mathcal{T}(\mathbf{r}|_1)} = \{ \text{root node} \} \rightarrow \widehat{S}_*$, where we require  $\phi \in \widehat{S}_{r_1}$. 
	\item If $k>1$, given an $\mathbf{r}|_k$-label $\phi: V_{\mathcal{T}(\mathbf{r}|_k)} \rightarrow \widehat{S}_*$, we require 
		\begin{enumerate}
		\item for any child $v$ of the root, the $\mathbf{r}|_{k-1}$-label 
		$\phi|_{\mathcal{T}_v}$ is in  $\mathcal{T}(\mathbf{r}_{k-1})$, and 
		\item $\phi(\text{root node}) \in \widehat{S}_\alpha$, where $S_\alpha$ gives the stabilizer of the action by $S_{r_k}$ on $\mathbf{r}|_{k-1}$-sublabels of $\phi$, so that $\alpha \vdash r_k$ is the partition of $[r_k]$ given by the number of $\mathbf{r}|_{k-1}$-sublabels of $\phi$ in each nonempty equivalence class, 
		\end{enumerate}		
		where $\phi|_{\mathcal{T}_v}$ denotes the restriction of $\phi$ to the subtree $\mathcal{T}_v$. 
	\end{enumerate}
\end{definition}

\subsection{Bratteli diagrams}\label{subsection:Bratteli}

We refer to Section 4.1 in \cite{MR2081042} or to \cite{maslen2001cooley} for a detailed discussion on Bratteli diagrams. 

A {\em Bratteli diagram} $B$ is a weighted graph, which may be described by a set of vertices from a disjoint collection of sets $B_k$, $k\geq 0$, and edges that connect vertices in $B_k$ to vertices in $B_{k+1}$. 
Assume that the set $B_0$ contains a unique vertex, and that the edges are labeled by positive integer weights. In the case the multiplicity is $1$, we omit the labels. The set $B_k$ is the set of vertices at level $k$. If a vertex  
$v_k\in B_k$ is connected to a vertex $v_{k+1}\in B_{k+1}$, then we write $v_k\leq v_{k+1}$. 

Given a tower of subgroups $\langle 1 \rangle=G_0\leq G_1 \leq \ldots \leq G_n$, the corresponding Bratteli diagram has vertices of set $B_i$ labeling the irreducible representations of $G_i$. 
If $\rho$ and $\eta$ are irreducible representations of $G_i$ and $G_{i-1}$, respectively, then the corresponding vertices are connected by an edge weighted by the multiplicity of $\eta$ in $\rho$ when restricted to the group $G_{i-1}$.

\begin{example}\label{example:Bratteli-diagram}
The Young lattice is an example of a Bratteli diagram, where the vertices represent Young diagrams, or partitions, and the edge joining a partition of $k$ to a partition of $k+1$ has weight $1$. For the symmetric group $S_4$ for the sequence $S_1<S_2<S_3<S_4$ of subgroups, where $S_i$ permutes only the symbols $1,\ldots,i$, the Bratteli diagram has the form: 
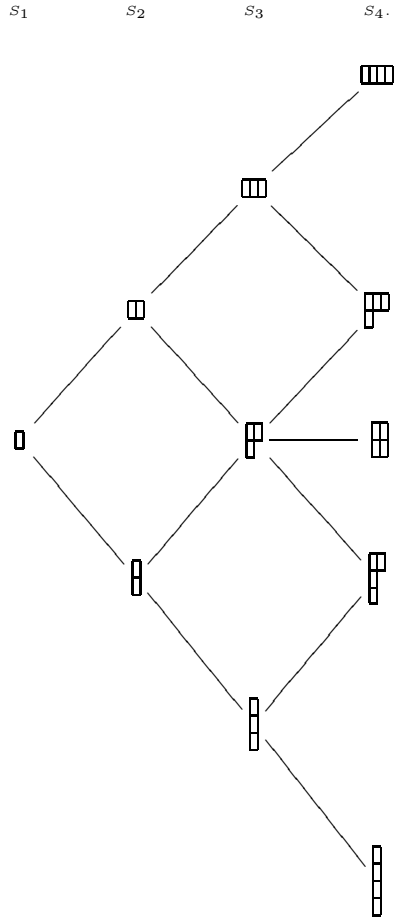
\tiny{ 
\[ 
\xymatrix@-1pc{
S_1& &S_2 & &S_3 & & S_4. \\ 
& &  & & & &\mbox{\begin{tabular}{|c|c|c|c|}
\hline  
 & &  & \\
\hline  
\end{tabular}}\\ 
& & & & & & \\ 
& & & & \mbox{\begin{tabular}{|c|c|c|}
\hline  
 & &  \\
\hline  
\end{tabular}} \ar@{-}[rrdd] \ar@{-}[rruu] & & \\ 
& & & & &  & \\ 
 &  &  \mbox{\begin{tabular}{|c|c|}
\hline  
 &  \\
\hline  
\end{tabular}}\ar@{-}[rruu] \ar@{-}[rrdd] & & & & \mbox{\begin{tabular}{|c|c|c|}
\hline  
 &  &  \\
 \hline \\
\cline{1-1} 
\end{tabular}} \\
& &  &  & &  & \\
\mbox{\mbox{ \begin{tabular}{|c|}
\hline  
   \\
\hline  
\end{tabular}} } \ar@{-}[rruu]\ar@{-}[rrdd] &  & & &\mbox{\begin{tabular}{|c|c|}
\hline  
 &   \\
\hline  \\
\cline{1-1} 
\end{tabular}}\ar@{-}[rruu]\ar@{-}[rrdd] \ar@{-}[rr]& & \mbox{ \begin{tabular}{|c|c|}
\hline  
 &   \\
\hline  
& \\
\hline 
\end{tabular}}  \\ 
& & & & & & \\
&  & \mbox{\begin{tabular}{|c|}
\hline  
   \\
\hline  
\\ 
\hline 
\end{tabular}}\ar@{-}[rruu] \ar@{-}[rrdd]& & & & \mbox{\begin{tabular}{|c|c|}
\hline  
 &   \\
\hline  \\
\cline{1-1} 
\cline{1-1}  \\
\cline{1-1} 
\end{tabular}} \\ 
& & & &  & &  \\ 
& & &  & \mbox{\begin{tabular}{|c|}
\hline  
   \\
\hline  
\\ 
\hline 
\\ 
\hline 
\end{tabular}} \ar@{-}[rrdd] \ar@{-}[rruu] &  & \\ 
& & & & & & \\ 
& & & & & & \mbox{\begin{tabular}{|c|}
\hline  
   \\
\hline  
\\ 
\hline 
\\ 
\hline 
\\
\hline 
\end{tabular}} \\ 
}
\] 
}
\normalsize 
 The distinct edges in the Bratteli diagram viewed as directed from level $i-1$ to level $i$ may be viewed as mutually orthogonal $S_{i-1}$-equivariant morphisms $\mathbb{C}[S_{i-1}]\rightarrow \mathbb{C}[S_i]$; the paths from the root to  a leaf give a natural indexing of Gelfand-Tsetlin bases for the towers of subgroups (cf. \cite{GT-bases-AnSn}, \cite{vershik2004new}). 
 These bases correspond to those matrix representations which are block diagonal with irreducible blocks at each step (with equivalent irreducibles being equal) when restricted through the tower of subgroups. 
\end{example}

\subsection{Adapted bases and fast Fourier transforms}\label{subsection:adaped-bases-FFT}    
 
Classical discrete Fourier transform (DFT) and fast Fourier transform (FFT) based approaches come from the use of commutative groups.  We expand the original work by Holmes (cf. \cite{Holmes-signal-processing}) and Karpovsky-Trachtenberg (cf. \cite{karpovsky1979fourier}) who merged DFT and FFT for signal processing to noncommutative groups. % In this section, we calculate FFT of iterated wreath product of finite groups. 

Let $\mathcal{R}$ be a set of traversals of $\widehat{G}$. 
We recall some foundational background from \cite{MR1339806}. 

\begin{definition}\label{definition:Fourier-transform-finite-group}   
Let $G$ be a finite group, and let $L(G)$ be the $|G|$-dimensional complex vector space of functions defined on $G$. If $\rho$ is a matrix representation of $G$, then the {\em Fourier transform $\widehat{f}(\rho)$ of $f$ at $\rho$} is the matrix sum    
\[    
\widehat{f}(\rho) = \sum_{g\in G} f(g)\rho(g).   
\]   
\end{definition}  
 
\begin{definition}\label{definition:discrete-Fourier-transform}  
The {\em discrete Fourier transform} $DFT_{\mathcal{R}}(f)$ with respect to a traversal $\mathcal{R}\subseteq \widehat{G}$  is the collection of individual Fourier transforms  
\[  
DFT_{\mathcal{R}}(f) = \left\{\widehat{f}(\rho): \rho\in \mathcal{R} \right\}.  
\]  
\end{definition}
The following notion of adapted bases is fundamental in the FFT algorithm.
\begin{definition} 
Let $H$ be a subgroup of a group $G$ and let 
$\mathcal{S}=\{ \eta_1,\ldots, \eta_l\}$ and 
$\mathcal{R}=\{ \rho_1,\ldots, \rho_h\}$ be sets of matrix representations for $H$ and $G$, respectively. 
Then the pair $(G,\mathcal{R})$ is $(H,\mathcal{S})$-{\em adapted} if for all $1\leq i\leq h$ and $y\in H$, 
\[ 
\rho_i(y) = \eta_{i_1}(y)\oplus \ldots \oplus \eta_{i_m}(y)
\] 
for some $\eta_{i_j}\in \mathcal{S}$. 
\end{definition}

Let $T(G,\mathcal{R})$ be the computational time to compute discrete Fourier transform for an arbitrary function $f$ with respect to a traversal $\mathcal{R}$. Let $T(G)$ be the minimum of $T(G,\mathcal{R})$ over all $\mathcal{R}$. 
We now cite a theorem:
\begin{theorem}[Theorem 3.1, \cite{MR1192969}] 
The Fourier transform for the symmetric group $S_n$ may be evaluated in no more than 
$\left(\dfrac{5}{12}n^3 +\dfrac{1}{2}n^2 -\dfrac{11}{12}n \right)n!$ arithmetic operations. 
\end{theorem}

\section{Irreducible representations of iterated wreath products}\label{section:irrep-iterated-wreath-products}

In this section, we will prove Theorem~\ref{thm:traversal}.

\begin{proof}
	Let $\{ \rho_1, \ldots, \rho_h \}$ be a traversal for $G$. For $\mathbf{i} \in [h]^n$, denote 
	$$
	\rho(\mathbf{i}) := \rho_{i_1} \boxtimes \rho_{i_2} \boxtimes \cdots \boxtimes \rho_{i_n}.
	$$ 
	
Let $\rho(\mathbf{i}) \in \mathcal{R}_{G^n}$ be fixed. 
Let $\sigma \in S_n$. 
The action of $\sigma$ on $\rho(\mathbf{i})$ is given by 
$$
\left(\rho(\mathbf{i})\right)^\sigma (g_1, \ldots , g_n) = (\rho_{i_1} \boxtimes \cdots \boxtimes \rho_{i_n})(g_{\sigma^{-1}(1)}, \ldots , g_{\sigma^{-1}(n)}) 
$$ 
by Definition~\ref{def:action-on-irreps}.  
Since $\left( \rho(\mathbf{i}) \right)^\sigma \sim \rho(\mathbf{i})$ if and only if $i_\ell = i_{\sigma(\ell)}$ for every $\ell \in [n]$ by Lemma~\ref{lemma:orbits-reps-of-GN}, 
the inertia group of $\rho(\mathbf{i})$ is 
$$
S_{A_1(\mathbf{i})} \times \cdots \times S_{A_h(\mathbf{i})} \cong S_\alpha,
$$ 
where 
$$
A_\ell(\mathbf{i}) = \{j \in [n] : i_j = \ell\} \subseteq [n] 
$$ 
and $\alpha_\ell = |A_\ell|$. 

So for $\rho^\alpha \in \mathcal{R}_{G^n}$ and $I = I_{G \wr S_n}(\rho^\alpha)$, 
the irreducible representation 
$\rho^\alpha$ has an extension to $I$ by the little group method (an application of Clifford theory) to the structure of irreducible representations of $G\wr S_n=G^n\rtimes S_n$. We thus find that a traversal for $G \wr S_n$ is precisely  
$$
		 \left\{ \Ind_{G\wr S_\alpha}^{G \wr S_n} 
		 ((\rho_1^{\boxtimes \alpha_1} \boxtimes \cdots \boxtimes \rho_h^{\boxtimes \alpha_h})' \otimes \sigma') 
		 : \alpha \models_h n, \sigma \in \mathcal{R}_{S_{\alpha}} \right\}. 
$$ 
The second statement of the theorem follows as a special case of the main statement. 
\end{proof}

Let $N(G)$ be the number of non-isomorphic irreducible representations of a finite group $G$. 
If we denote by $\mathcal{P}(n)$ the number of partitions of the integer $n$ (so that  $\mathcal{P}(n) = | \{ \alpha: \alpha \vdash n \}|$), where $\mathcal{P}(0):=1$, then  $N(S_n) = \mathcal{P}(n)$.

\begin{corollary}[Lemma 4.2.9, \cite{MR644144}]\label{cor:recursion-number-of-irrep-GSn}
Suppose that $N(G) = h$. Then the number of non-isomorphic irreducible representations of $G \wr S_n$ is given by 
	\begin{equation}
	N(G \wr S_n) = \sum\limits_{\alpha \models_h n} \prod_{i \in [h]} \mathcal{P}(\alpha_i). 
	\end{equation}
\end{corollary}
 
\begin{proof}
This follows from Theorem~\ref{thm:traversal} and an application of Clifford theory. 
\end{proof}

Let $N(\mathbf{r}|_k)$ denote the number of equivalence classes of ordinary irreducible representations for the wreath product $W(\mathbf{r}|_k)=W(\mathbf{r}|_{k-1})\wr S_{r_k}$.  
Define $P(n,h) : =  \sum_{\alpha \models_h n} \prod_{i=1}^h \mathcal{P}(\alpha_i)$.  
\begin{corollary}\label{cor:recursion-number-of-irrep}
It follows that for $\alpha \vdash n$, 
	\begin{equation}
	N(G \wr S_\alpha) = \prod\limits_{j=1}^{|\alpha|} \sum\limits_{\beta \models_h \alpha_j} \prod_{i \in [h]} \mathcal{P}(\beta_i). 
	\end{equation}
For $h:=N(\mathbf{r}|_{k-1})$, 
we find that $h$ satisfies the following recursion:  
	\begin{equation}
	N(\mathbf{r}|_k) = P\left( r_k, N(\mathbf{r}|_{k-1}) \right) 
	= \sum_{\alpha\models_h r_k } \prod_{i\in [h]}  \mathcal{P}(\alpha_i). 
	\end{equation}
\end{corollary}

\begin{proof}
This follows from Corollary~\ref{cor:recursion-number-of-irrep-GSn}. 
\end{proof}

\section{Branching diagram and $\mathbf{r}$-label correspondence}\label{section:branching-diagram}
 
We find a combinatorial structure describing the branching diagrams for iterated wreath products of symmetric groups by proving Theorem~\ref{theorem:bijection-tree-iterated-wreath}.

\begin{proof}
It suffices to define a map on a traversal of $\widehat{W}(\mathbf{r}|_k)$, which is given in \eqref{eq_set-of-irreps}. We will define a map $F: \mathcal{R}_{W(\mathbf{r}|_k)} \rightarrow \mathcal{T}( \mathbf{r}|_k)$ recursively, and it suffices to prove that each orbit of $\mathcal{T}(\mathbf{r}|_k)$ under action by $W(\mathbf{r}|_k)$ has exactly one pre-image under $F$.  

Let $k=1$. For  any $\rho \in \widehat{W}(\mathbf{r}|_k) = \widehat{S}_{r_1} $, we define the $\mathbf{r}|_1$-label as 
$$
F(\rho): V_{\mathcal{T}(\mathbf{r}|_1)} = \{\text{root}\} \rightarrow \widehat{S}_{r_1},\hspace{4mm}
 \mbox{ where } F(\rho) (\text{root}) := \rho.
$$   
This is clearly a bijection as desired. 

Now let $k>1$. By the inductive hypothesis, $F: \mathcal{R}_{W(\mathbf{r}|_{k-1})} \rightarrow \mathcal{T}( \mathbf{r}|_{k-1})$ has exactly one pre-image per orbit of $\mathcal{T}(\mathbf{r}|_k)$. Suppose that a traversal for $W(\mathbf{r}|_{k-1})$ is given by the set $\{\rho_1, \ldots, \rho_h \}$. We need to define $F$ on $\mathcal{R}_{W(\mathbf{r}|_k)}$, and show that orbits have exactly one pre-image as desired. 

Pick an arbitrary element of $\rho_1^{\alpha_1} \otimes \cdots \otimes \rho_h^{\alpha_h} \otimes \sigma$ of $\mathcal{R}_{W(\mathbf{r}|_k)}$. Denote its image under $F$ by 
$$
\phi :=F(\rho_1^{\alpha_1} \otimes \cdots \otimes \rho_h^{\alpha_h} \otimes \sigma): V_{\mathcal{T}(\mathbf{r}|_k)} \rightarrow S_{r_k}.
$$  
Let $U \subset V_{\mathcal{T}(\mathbf{r}|_k)}$ be the $r_k$ children of the root. Assign an ordering to $U= \{u_1,\ldots, u_{r_k} \}$. 
Then partition the set $U$ as 
$$
U = U_1 \sqcup \ldots \sqcup U_h,
$$  
where each $U_i$ satisfies $|U_i| = \alpha_i$ while preserving the ordering. 
For each $u_i \in U$, define the value of $\phi$ on all nodes in subtree $\mathcal{T}_{u_i}$ to satisfy $\phi|_{\mathcal{T}_{u_i}} := F(\rho_{j^i})$, where $j^i$ satisfies $U_{j^i} \ni u_i$  and where  
$\phi|_{\mathcal{T}_{u_i}}$ denotes the restriction of $\phi$ to the subtree 
$\mathcal{T}_{u_i} \subseteq \mathcal{T}$.  It remains to define the value of $\phi$ on the root node. We let $\phi(\text{root}) = \sigma$.  

Notice that $\phi|_{\mathcal{T}_{u_i}} \in \mathcal{T}(\mathbf{r}|_{k-1})$ by definition and induction. Since $\sigma$ is in the stabilizer of the $S_{r_k}$-action on $\rho^\alpha$, which is exactly $S_\alpha$, we see that $\phi$ is a compatible label for $\mathcal{T}(\mathbf{r}|_k)$. Thus, $F$ is well-defined, and each orbit of $\mathcal{T}(\mathbf{r}|_k)$ has exactly one pre-image. 
\end{proof}

\subsection{Degrees of irreducible representations}\label{section:degrees}

Following the discussion in Section 4.1.1 in \cite{MR2081042}, we define for any $\mathbf{r}|_k$-tree $\mathcal{T}$ the companion tree $C_{\mathcal{T}}$. 

\begin{definition}
Fix $\mathcal{T}(\mathbf{r}|_k)$ and $\mathbf{r}|_k$-label $\phi$. Let the {\em companion label $C_\phi : V_{\mathcal{T}(\mathbf{r}|_k)} \rightarrow \mathbb{N}$} be defined by: 

\begin{equation*}
C_\phi(v) = \begin{cases}
\dim(\phi(v)) & \text{ if $v$ is a leaf of the tree $\mathcal{T}(\mathbf{r}|_k)$,} \\
|S_{r_i}/S_\alpha| = {r_i \choose \alpha} & \text{ otherwise, where $v$ is in the $(k-i)$-th layer of }\\
& \text{ $\mathcal{T}$ and $\phi(v) \in S_\alpha$}. \\
\end{cases}
\end{equation*}
\end{definition}

Similar to Proposition 4.3 in \cite{MR2081042}, we obtain the following: 

\begin{proposition}\label{prop:dimension-irreducible-rep}
	
If $\rho$ is an irreducible representation of $W(\mathbf{r}|_k)$ associated to $\mathbf{r}|_k$-label $\phi$, then the dimension $d_\rho$ of $\rho$ is given by 

        \begin{equation}
        d_\rho = \prod_{v} C_\phi(v), 
        \end{equation}
the product of the value of the companion label $C_\phi$ on all vertices. 

\end{proposition}

\section{Fast Fourier transforms, adapted bases and upper bound estimates}\label{section:FFT}   
We use the FFT estimates derived in \cite{MR1192969} and \cite{MR1339806}  
to state a coarse, overall upper bound on the running time of FFT for the wreath product $W(\mathbf{r}|_k)$. 
\begin{theorem}[Theorem 3, \cite{MR1339806}]\label{thm:Roc95-FFT} 
We have 
	\begin{equation}
T(G \wr S_n) \leq n T(G) \cdot |G \wr S_{n-1}| + n T( G \wr S_{n-1}) \cdot |G| + n^3 2^{|\mathcal{R}_G|} |G \wr S_n|.
	\end{equation}
\end{theorem}

The separation of variables approach has been one of the primarily components that is responsible for the fastest known algorithms for almost all classes of finite groups, including symmetric groups 
\cite{MR1468943} and their wreath products \cite{MR0178586}.

\begin{corollary}\label{cor:FFT-symmetric-groups}
Let $T(\mathbf{r}|_k)$ be the computation time for the wreath product $W(\mathbf{r}|_k)$.  
Then 
	\begin{equation}
	\begin{split}
T(\mathbf{r}|_k) &\leq 
r_k \prod_{i=1}^{k-1} (r_i!) 
\bigg(
 (r_{k-1}!)
 \left( 
   T(\mathbf{r}|_{k-1}) 
 + r_k^3 
2^{|\mathcal{R}_{W(\mathbf{r}|_{k-1})} |}
\right) 
  \\
&\hspace{4mm}  
+ T(W(\mathbf{r}|_{k-1}) \wr S_{r_{k-1}} )   
\bigg). \\ 
\end{split}
	\end{equation}
\end{corollary}

\begin{proof}
This result is a consequence of Theorem~\ref{thm:Roc95-FFT}. 
\end{proof}

\section{Conclusion and open problems}\label{section:conclusion-open-problem}

As a sequel to \cite{Branching-diag-cyclic-Im-Wu}, 
we have given an explicit description of a traversal for the iterated wreath product $W(\mathbf{r}|_k)$ and we have shown the existence of a bijection between equivalence classes of ordinary irreducible representations of the generalized iterated wreath products and $W(\mathbf{r}|_k)$-orbits of complete rooted trees. We have also stated a recursion for the number of equivalence classes of ordinary irreducible representations of the iterated wreath product, and have given the dimension of an irreducible representation of $W(\mathbf{r}|_k)$.  

We conclude by giving several open problems. One problem is to find a tighter fast Fourier transform (FFT) bound for chains of subgroups of $W(\mathbf{r}|_k)$ than the upper bound stated in Theorem~\ref{thm:Roc95-FFT}. Another problem is to  study the representation theory of, and find adapted bases and FFT operation bounds for chains of subgroups  for, iterated wreath products of more general class of groups.

\appendix

\newcommand{\etalchar}[1]{$^{#1}$}
\def\cprime{$'$} \def\cprime{$'$} \def\cprime{$'$} \def\cprime{$'$}
\providecommand{\bysame}{\leavevmode\hbox to3em{\hrulefill}\thinspace}
\providecommand{\MR}{\relax\ifhmode\unskip\space\fi MR }
% \MRhref is called by the amsart/book/proc definition of \MR.
\providecommand{\MRhref}[2]{%
  \href{http://www.ams.org/mathscinet-getitem?mr=#1}{#2}
}
\providecommand{\href}[2]{#2}


\begin{thebibliography}{FHM{\etalchar{+}}99b}

\bibitem{Stankovic-Moraga-Astola}
Jaakko~T. Astola, Claudio Moraga, and Radomir~S. Stankovi\'c, \emph{Fourier
  analysis on finite groups with applications in signal processing and system
  design}, John Wiley \& Sons, Inc (2005).

\bibitem{balasubramanian1979enumeration}
K~Balasubramanian, \emph{Enumeration of internal rotation reactions and their
  reaction graphs}, Theoretica chimica acta \textbf{53} (1979), no.~2,
  129--146.

\bibitem{MR585739}
K.~Balasubramanian, \emph{Graph theoretical characterization of {NMR} groups,
  nonrigid nuclear spin species and the construction of symmetry adapted {NMR}
  spin functions}, J. Chem. Phys. \textbf{73} (1980), no.~7, 3321--3337.

\bibitem{borsa2015wreath}
Diana Borsa, Thore Graepel, and Andrew Gordon, \emph{The wreath process: A
  totally generative model of geometric shape based on nested symmetries},
  arXiv preprint arXiv:1506.03041 (2015).

\bibitem{MR1192969}
Michael Clausen and Ulrich Baum, \emph{Fast {F}ourier transforms for symmetric
  groups: theory and implementation}, Math. Comp. \textbf{61} (1993), no.~204,
  833--847.

\bibitem{Chang-Thesis}
Will Chang, \emph{Image processing with wreath product groups},
  \url{https://www.math.hmc.edu/seniorthesis/archives/2004/wchang/wchang-2004-thesis.pdf},
  2004.

\bibitem{MR3625572}
Karl-Dieter Crisman and Michael~E. Orrison, \emph{Representation theory of the
  symmetric group in voting theory and game theory}, Algebraic and geometric
  methods in discrete mathematics, Contemp. Math., vol. 685, Amer. Math. Soc.,
  Providence, RI, 2017, pp.~97--115.

\bibitem{MR0202859}
A.~J. Coleman, \emph{Induced representations with applications to {$S_{n}$} and
  {${\rm GL}(n)$}}, Lecture notes prepared by C. J. Bradley. Queen's Papers in
  Pure and Applied Mathematics, No. 4, Queen's University, Kingston, Ont.,
  1966.

\bibitem{MR1038525}
Charles~W. Curtis and Irving Reiner, \emph{Methods of representation theory.
  {V}ol. {I}}, Wiley Classics Library, John Wiley \& Sons, Inc., New York,
  1990, With applications to finite groups and orders, Reprint of the 1981
  original, A Wiley-Interscience Publication.

\bibitem{MR2760311}
T.~Ceccherini-Silberstein, F.~Scarabotti, and F.~Tolli, \emph{Clifford theory
  and applications}, J. Math. Sci. (N.Y.) \textbf{156} (2009), no.~1, 29--43,
  Functional analysis.

\bibitem{MR2643487}
Tullio Ceccherini-Silberstein, Fabio Scarabotti, and Filippo Tolli,
  \emph{Representation theory of the symmetric groups}, Cambridge Studies in
  Advanced Mathematics, vol. 121, Cambridge University Press, Cambridge, 2010,
  The Okounkov-Vershik approach, character formulas, and partition algebras.

\bibitem{MR3202374}
\bysame, \emph{Representation theory and harmonic analysis of wreath products
  of finite groups}, London Mathematical Society Lecture Note Series, vol. 410,
  Cambridge University Press, Cambridge, 2014.

\bibitem{MR0178586}
James~W. Cooley and John~W. Tukey, \emph{An algorithm for the machine
  calculation of complex {F}ourier series}, Math. Comp. \textbf{19} (1965),
  297--301.

\bibitem{MR2572103}
Zajj Daugherty, Alexander~K. Eustis, Gregory Minton, and Michael~E. Orrison,
  \emph{Voting, the symmetric group, and representation theory}, Amer. Math.
  Monthly \textbf{116} (2009), no.~8, 667--687.

\bibitem{MR1153249}
William Fulton and Joe Harris, \emph{Representation theory}, Graduate Texts in
  Mathematics, vol. 129, Springer-Verlag, New York, 1991, A first course,
  Readings in Mathematics.

\bibitem{Mirchandani99multiresolution-analysis}
R.~Foote, D.~Healy, G.~Mirchandani, T.~Olson, and D.~Rockmore, \emph{A wreath
  product group approach to signal and image processing: Part {I} -
  multiresolution analysis}, 1999.

\bibitem{Mirchandani99awreath}
\bysame, \emph{A wreath product group approach to signal and image processing:
  Part {II} - convolution, correlation, and applications}, 1999.

\bibitem{GT-bases-AnSn}
T.~Geetha and Amritanshu Prasad, \emph{Comparison of {G}elfand-{T}setlin bases
  for alternating and symmetric groups}, arXiv preprint arXiv:1606.04424
  (2017).

\bibitem{Holmes-mathematical-foundations}
Richard~B. Holmes, \emph{Mathematical foundations of signal processing {II}.
  the role of group theory}, MIT Lincoln Laboratory, Lexington, MA
  \textbf{Technical Report 781} (1987), 1--97.

\bibitem{Holmes-signal-processing}
\bysame, \emph{Signal processing on finite groups}, MIT Lincoln Laboratory,
  Lexington, MA \textbf{Technical Report 873} (1990), 1--38.

\bibitem{Branching-diag-cyclic-Im-Wu}
Mee~Seong Im and Angela Wu, \emph{Generalized iterated wreath products of
  cyclic groups and rooted trees correspondence},
  \url{https://arxiv.org/abs/1409.0603}, to appear in Adv. Math. Sci. 

\bibitem{MR644144}
Gordon James and Adalbert Kerber, \emph{The representation theory of the
  symmetric group}, Encyclopedia of Mathematics and its Applications, vol.~16,
  Addison-Wesley Publishing Co., Reading, Mass., 1981, With a foreword by P. M.
  Cohn, With an introduction by Gilbert de B. Robinson.

\bibitem{MR1060103}
Gregory Karpilovsky, \emph{Clifford theory for group representations},
  North-Holland Mathematics Studies, vol. 156, North-Holland Publishing Co.,
  Amsterdam, 1989, Notas de Matem\'atica [Mathematical Notes], 125.

\bibitem{MR0325752}
Adalbert Kerber, \emph{Representations of permutation groups. {I}}, Lecture
  Notes in Mathematics, Vol. 240, Springer-Verlag, Berlin-New York, 1971.

\bibitem{kleshchev2010representation}
Alexander Kleshchev, \emph{Representation theory of symmetric groups and
  related {H}ecke algebras}, Bulletin of the American Mathematical Society
  \textbf{47} (2010), no.~3, 419--481.

\bibitem{karpovsky1979fourier}
Mark~G Karpovsky and EA~Trachtenberg, \emph{Fourier transform over finite
  groups for error detection and error correction in computation channels},
  Information and Control \textbf{40} (1979), no.~3, 335--358.

\bibitem{Lee-Stephen}
Stephen Lee, \emph{Understanding voting for committees using wreath products},
  \url{https://www.math.hmc.edu/seniorthesis/archives/2010/slee/slee-2010-thesis.pdf},
  2010.

\bibitem{leyton2003generative}
Michael Leyton, \emph{A generative theory of shape}, vol. 2145, Springer, 2003.

\bibitem{MR1468943}
David~K. Maslen, \emph{The efficient computation of {F}ourier transforms on the
  symmetric group}, Math. Comp. \textbf{67} (1998), no.~223, 1121--1147.

\bibitem{milot2001energy}
Robin Milot, AW~Kleyn, and APJ Jansen, \emph{Energy dissipation and scattering
  angle distribution analysis of the classical trajectory calculations of
  methane scattering from a {N}i (111) surface}, The Journal of Chemical
  Physics \textbf{115} (2001), no.~8, 3888--3894.

\bibitem{maslen2001cooley}
David~K Maslen and Daniel~N Rockmore, \emph{The {C}ooley-{T}ukey {FFT} and
  group theory}, Notices of the AMS \textbf{48}, no.~10, 1151--1160.

\bibitem{MR2081042}
R.~C. Orellana, M.~E. Orrison, and D.~N. Rockmore, \emph{Rooted trees and
  iterated wreath products of cyclic groups}, Adv. in Appl. Math. \textbf{33}
  (2004), no.~3, 531--547.

\bibitem{MR1339806}
Daniel~N. Rockmore, \emph{Fast {F}ourier transforms for wreath products}, Appl.
  Comput. Harmon. Anal. \textbf{2} (1995), no.~3, 279--292.

\bibitem{schnell2010understanding}
Melanie Schnell, \emph{Understanding high-resolution spectra of nonrigid
  molecules using group theory}, ChemPhysChem \textbf{11} (2010), no.~4,
  758--780.

\bibitem{MR1363490}
Barry Simon, \emph{Representations of finite and compact groups}, Graduate
  Studies in Mathematics, vol.~10, American Mathematical Society, Providence,
  RI, 1996.

\bibitem{vershik2004new}
Anatolii~Moiseevich Vershik and Andrei~Yur'evich Okounkov, \emph{A new approach
  to the representation theory of the symmetric groups. {II}}, Zapiski
  Nauchnykh Seminarov POMI \textbf{307} (2004), 57--98.

\end{thebibliography}
\end{document}